\documentclass[12pt]{amsart}

\usepackage[a4paper, hmargin=2.5cm, vmargin={3cm, 3cm}]{geometry}

\usepackage{amssymb, amsthm, amsmath}
\usepackage{xcolor}
\usepackage{graphicx}

\newcommand*\mcupinn[2]{\vcenter{\hbox{$\mathsurround=0pt
  \ifx\displaystyle#1\textstyle\else#1\fi\bigcup$}}}
\newtheorem{theorem}{Theorem}[section]

\DeclareGraphicsExtensions{.pdf,.png,.jpg}
\usepackage{pict2e}

\DeclareFontFamily{U} {MnSymbolC}{}

\DeclareFontShape{U}{MnSymbolC}{m}{n}{
  <-6> MnSymbolC5
  <6-7> MnSymbolC6
  <7-8> MnSymbolC7
  <8-9> MnSymbolC8
  <9-10> MnSymbolC9
  <10-12> MnSymbolC10
  <12-> MnSymbolC12}{}
\DeclareFontShape{U}{MnSymbolC}{b}{n}{
  <-6> MnSymbolC-Bold5
  <6-7> MnSymbolC-Bold6
  <7-8> MnSymbolC-Bold7
  <8-9> MnSymbolC-Bold8
  <9-10> MnSymbolC-Bold9
  <10-12> MnSymbolC-Bold10
  <12-> MnSymbolC-Bold12}{}

\DeclareSymbolFont{MnSyC} {U} {MnSymbolC}{m}{n}
\DeclareMathSymbol{\diamondtimes}{\mathbin}{MnSyC}{125}

\newtheorem{lemma}[theorem]{Lemma}
\newtheorem{corollary}[theorem]{Corollary} 
\newtheorem{prop}[theorem]{Proposition}

\newtheorem{obs}{Observation}

\newtheorem*{theorem*}{Theorem}{\bf}{\it}
\newtheorem*{proposition*}{Proposition}{\bf}{\it}
\newtheorem*{observation*}{Observation}{\bf}{\it}
\newtheorem*{lemma*}{Lemma}{\bf}{\it}

\theoremstyle{definition}
\newtheorem{definition}[theorem]{Definition}

\theoremstyle{remark}
\newtheorem{remark}[theorem]{Remark}

\newcommand{\Z}{\mathbb Z}
\newcommand{\R}{\mathbb R}

\newcommand{\C}{\mathbb C}

\def\XXint#1#2#3{{\setbox0=\hbox{$#1{#2#3}{\int}$ }
\vcenter{\hbox{$#2#3$ }}\kern-.6\wd0}}

\begin{document}
\title[]{A discrete harmonic function bounded on a large portion of $\Z^2$ is constant}

\author[L. Buhovsky]{Lev Buhovsky }
\address{L.B.: School of Mathematical Sciences, Tel Aviv University, Tel Aviv 69978, Israel}
\email{levbuh@post.tau.ac.il}
\thanks{L.B. supported in part by ISF Grant 1380/13,
and by the Alon Fellowship}

\author[A. Logunov]{Alexander Logunov}
\address{A.L.:School of Mathematics, Institute for Advanced Study, Princeton, NJ 08540;}
\address{ School of Mathematical Sciences, Tel Aviv University, Tel Aviv 69978, Israel;}
\address{Chebyshev Laboratory, St. Petersburg State University, 14th Line V.O., 29B, Saint Petersburg 199178 Russia}
\email{log239@yandex.ru}
\thanks{A.L. supported  in part by ERC Advanced Grant~692616 and ISF Grants~1380/13, 382/15}

\author[E. Malinnikova]{Eugenia Malinnikova}
\address{E.M.: Department of Mathematical Sciences,
Norwegian University of Science and Technology
7491, Trondheim, Norway}
\email{eugenia.malinnikova@ntnu.no}
\thanks{E.M. supported in part by  Project 213638
of the Research Council of Norway}

\author[M. Sodin]{Mikhail Sodin}
\address{M.S.: School of Mathematical Sciences, Tel Aviv University, Tel Aviv 69978, Israel}
\email{sodin@post.tau.ac.il}
\thanks{M.S. supported in part by ERC Advanced Grant~692616 and ISF Grant~382/15}

\dedicatory{To Fedya Nazarov with admiration}

\begin{abstract} An improvement of the Liouville theorem for discrete harmonic functions on $\mathbb{Z}^2$ is obtained. More precisely, we prove that there exists a positive constant $\varepsilon$ such that if $u$ is discrete harmonic  on $\mathbb{Z}^2$ and for each sufficiently large square $Q$ centered at the origin $|u|\le 1$  on a $(1-\varepsilon)$ portion of $Q$ then $u$ is constant.
 \end{abstract}
\maketitle
\section{Introduction}
Let  $u$  be a discrete harmonic function on the lattice $\mathbb{Z}^2$, i.e., a function satisfying the mean value property: the value of $u$ at any point of $\mathbb{Z}^2$ is equal to the average of the four values at the adjacent points.
 
The Liouville theorem  states that if $u$ is bounded on   $\mathbb{Z}^2$, say $|u|\le 1$ everywhere, then $u\equiv const$. This statement  is classical \cite{C},\cite[Theorem 5]{H} and well known.

In the present  paper we obtain a somewhat unexpected  improvement of the Liouville theorem. We show that if $|u|$ is bounded  on $(1-\varepsilon)$ portion  of $\mathbb{Z}^2$, where $\varepsilon>0 $ is a sufficiently small numerical constant,  then $u$ is a constant function. The precise statement is given below.

\subsection{Main result.}
We  partition  the plane $\mathbb{R}^2$ into  unit squares (cells) so that  the centers of squares have integer coordinates and  identify the cells with the elements of $\mathbb{Z}^2$.

 Given a positive integer $N$, we denote by $Q_N$ the  square on $\mathbb{Z}^2$ with center at the origin and of size $(2N+1)\times (2N+1)$: $$Q_N:=\{(n,m): |n|,|m| \leq N \}.$$ The translation of this square by a vector $x$ with integer coordinates is denoted by $Q_N(x)$.
 For a set $S \subset \mathbb{Z}^2$, we denote by $|S|$ the number of elements  in $S$.

  Let $\varepsilon \in (0,1)$ be a positive number.
 We  say that $|u|$ is bounded by $1$ on $(1-\varepsilon)$ portion of the lattice $\mathbb{Z}^2$ if  for some $N_0>0$
 $$| Q_N \cap \{|u| \leq 1\} | \geq (1-\varepsilon) |Q_N| $$
for all $N \geq N_0$.

\begin{theorem}
\label{Main}
 There exists $\varepsilon>0$ such that if $u$ is a harmonic function on $\mathbb{Z}^2$ and $|u|$ is bounded by $1$  on $(1-\varepsilon)$ portion of  $\mathbb{Z}^2$, then $u$ is  constant.
\end{theorem}

 \begin{remark}
There is no continuous analog of Theorem \ref{Main} for harmonic
functions in $\mathbb R^2$. For instance, let $\Pi$ denote the
semi-strip
$\Pi=\bigl\{z\colon \text{Re}(z)>0, \  |\text{Im}(z)|<\pi  \bigr\}$,
and let $L=\partial\Pi$ be oriented so that $\Pi$ is on the right-hand side
of $L$. Then the integral
\[
E(z) =
\frac1{2\pi {\rm i}} \int_L \frac{e^{e^\zeta}}{\zeta-z}\,
\text{d}\zeta\,, \qquad z\in\mathbb C\setminus \Pi,
\]
has an analytic continuation on $\mathbb C$, which is bounded
outside $\Pi$ (see, for instance, \cite[Ch 3,
Problems 158-160]{PS}). Obviously, the harmonic function $H(z)=\text{Re}E(z)$ is
also bounded outside $\Pi$.

Furthermore, given an arbitrarily narrow curvilinear semi-strip
$\Pi$, symmetric with respect to the real axis and such that the
intersection of $\Pi$ with any vertical line consists of an open
interval, replacing the function $e^\zeta$ by another analytic function,
one can modify this construction to get an entire (and then, harmonic) function bounded outside $\Pi$. See, for instance the discussion in~[2].
\end{remark}

\begin{remark} The following simple example shows that the statement of Theorem \ref{Main} cannot be extended to higher dimensions. First, we consider the function $u_0:\mathbb{Z}^2\to \R$ defined by $u_0(x,y)=0$ when $x\neq y$ and $u_0(x,x)=(-1)^x$. This function  is not harmonic  but is an eigenfunction of the discrete Laplace operator,
\[\Delta u_0(x,y)=u_0(x+1,y)+u_0(x,y+1)+u_0(x-1,y)+u_0(x,y-1)-4u_0(x,y)=-4u_0(x,y).\]
Then we define a  function on $\mathbb{Z}^3$  by
\[u(x,y,z)=c^{z}u_0(x,y),\quad{\text{where}}\ \ c+c^{-1}=6,\]
and check that $u$ is a non-zero harmonic function on $\mathbb{Z}^3$ that vanishes everywhere except for the hyperplane $\{(x,y,z): x=y\}$.
\end{remark}

\subsection{Toy question and two examples}
A simpler uniqueness question can be asked in  connection to Theorem \ref{Main}. 
Let a discrete harmonic function $u$ be equal to zero on $(1-\varepsilon)$ portion of $\mathbb{Z}^2$. Does $u$ have to be zero identically?

 Theorem \ref{Main} implies the affirmative answer to that question if $\varepsilon>0$ is sufficiently small. On the other hand the statement is wrong for $\varepsilon=1/2$. One can construct a non-zero discrete harmonic function $u$ on $\mathbb{Z}^2$, which is equal to zero on half of $\mathbb{Z}^2$. Namely, $u$ may have zero values on a half-plane $ \{ (x,y) \in \mathbb{Z}^2: x-y \geq 0 \}$ without being zero everywhere. It is not difficult to see that on each next diagonal we can choose one value and then the rest of the values are determined uniquely, the details of such construction are given in Section \ref{ss:obs}.

Going back to the assumption of Theorem \ref{Main}, we note that there is a discrete harmonic function, which is bounded on $3/4$ of $\mathbb{Z}^2$. The following simple example was drawn to authors' attention by Dmitry Chelkak.
 Let $$u(n,m)= \sin(\pi n/2)e^{bm},$$ where $b$ is the positive solution of $e^b+e^{-b}=4$.
 It is easy to check that $u$ is a discrete harmonic function and $|u|$ is bounded by $1$ on $(2\Z\times\Z)\cup(\Z\times\Z_-)$.
 
We don't know the precise value of $\varepsilon$ for either the uniqueness question  or the boundedness question. One may also ask if those constants are equal. 

\subsection{Two theorems}
 The proof of Theorem \ref{Main} is based on two statements, which compete with each other.

\begin{theorem*}[{\bf{A}}]  
 For all sufficiently small $\varepsilon>0$,   there exist $a=a(\varepsilon)>0$ and $N_0=N_0(\varepsilon)$ such that  if   
\[| Q_N \cap \{|u| \leq 1\} | \geq (1-\varepsilon) |Q_N|,\quad N>N_0(\varepsilon),\] 
then
$$ \max\limits_{Q_{[N/2]}} |u| \leq e^{aN}. $$
Moreover,  $a(\varepsilon) \to 0$ as $\varepsilon \to 0$.

\end{theorem*}

\begin{theorem*}[{\bf{B}}] There exists $b>0$ such that the following holds.
If $\varepsilon>0$ is sufficiently small, $N$ is sufficiently large,  $\max\limits_{Q_{[\sqrt N]}} |u| \geq 2$
and 
$$| Q_K \cap \{|u| \leq 1\} | \geq (1-\varepsilon) |Q_K| $$
for every $K \in [\sqrt N , N]$, 
 then 
$$\max\limits_{Q_{[N/2]}} |u| \geq  e^{bN}.$$

\end{theorem*}

 \subsection{ Theorem \ref{Main} follows from Theorems (A) and (B)} 
Assume  that $|u|$ is bounded by $1$  on $(1-\varepsilon)$ portion of  $\mathbb{Z}^2$ and $\varepsilon$ is small enough so that the value of $a$ from Theorem (A) is strictly smaller than the value of $b$ from Theorem (B). If $u$ is not constant, then by the Liouville theorem  $\max\limits_{Q_{[\sqrt N]}} |u| \geq 2$ for some large $N$.
Then the conclusion of Theorem (A) contradicts  the conclusion of Theorem (B). 

\begin{remark} 
 As we have mentioned above, there  is no continuous version of Theorem (A), i.e., Theorem (A) is  essentially of discrete nature. On the other hand,  Theorem (B)
 has continuous analogs, which are  related to a version of the Hadamard three circle theorem  for harmonic functions. In fact, a stronger version 
 of Theorem (B) holds for continuous harmonic functions. If $f$ is a non-constant harmonic function on $\mathbb{R}^2$, then, informally speaking,   the larger the set $\{|f|<1\}$ is, the faster $\max_{|z|=R}|f(z)|$  grows as $R\to\infty$.  
\end{remark}  

 It is worth mentioning that we don't know whether  the statement of Theorem \ref{Main} still holds under the weaker assumption that $$|Q_N\cap\{|u|<1\}|\le(1-\varepsilon)|Q_N|\quad{\text{ for\ infinitely\ many}}\ N.$$
Another open question is whether Theorem \ref{Main} remains true if one replaces the boundedness  
by positivity of $u$ on a large portion of $\mathbb{Z}^2$ (for $\varepsilon=0$ the result is true and classical \cite{C}).

\section*{Acknowledgments} 
This work was motivated by the question raised by Lionel Levine
and Yuval Peres about existence of non-trivial translation-invariant
probability measures on harmonic functions on lattices. The discussion
with Dmitry Chelkak was also very helpful. The authors thank
all of them.

\section{Lower bound}
 This section is devoted to the proof of Theorem (B).

\subsection{Discrete three circle (square) theorem}
  We will use a discrete version of the Hadamard three circle theorem. Let $u$ be a discrete harmonic function on $Q_N$.
  Given a square $Q_K$, we  say that $u$ is bounded by $1$ on a half of $Q_K$ if 
 $$| Q_K \cap \{|u| \leq 1\} | \geq \frac{1}{2}|Q_K|.$$ 
\begin{theorem} \label{th:three}
 If $|u|\le 1$ on a half of  $Q_{[N/4]}$ and $|u|\le M$ on $Q_{N}$, then $$ \max_{Q_{[N/2]}}|u| \leq C M^{\alpha} + C e^{-cN} M,$$
 where $\alpha \in (0,1)$ and $c,C>0$ are numerical constants.
\end{theorem}

This is a generalization of the discrete three circle theorem proved in \cite{GM}, where the estimate $|u|\le 1$ was assumed on the whole square $Q_{[N/4]}$. The continuous analog of Theorem 2.1 is known, a more general statement can be found in \cite{N1}.
 Comparing the discrete theorem with the continuous one, we see that an error term $C e^{-cN} M$ appears at the RHS. This term has a discrete nature, cannot be removed and makes the three circle theorem in the discrete case weaker than in the continuous case. 

 The proof of Theorem \ref{th:three} can be obtained by a modification of the proof given in \cite{GM}. A simpler approach 
 is  outlined in the Appendix for the convenience of the reader. 


\subsection{Proof of Theorem (B)}
 We assume that $N$ is sufficiently large and $\varepsilon$ is sufficiently small. More precisely, if $\alpha$ is as in Theorem \ref{th:three}, we first choose a large constant $A$ such that $((1+\alpha)/2)^{A}<2^{-5}$ and then take $\varepsilon<\frac{1}{100A^2}$. The choice of the parameters will be justified later. 

For each integer $K$, $1\le K\le N$, denote  by $M(K)$ the maximum of $u$ over $Q_K$. 
 The following proposition is the main step in the proof of Theorem (B). 
 \begin{prop} \label{pr: M}
 In assumptions of Theorem (B), there exist $C_1,c_1>0$ such that if $M(K)>C_1$ and $K\in[\sqrt N, N/4]$, then 
\begin{equation}\label{eq:propM} 
 M(2K) \geq \min(M^{32}(K), M(K) e^{c_1K}).
\end{equation}
 
\end{prop} 
  The constant $32=2^5$ depends on our choice of the parameters, it can be replaced by any other constant provided that $\varepsilon$ is sufficiently small.

 \begin{proof} First, by the assumption, 
$$| Q_{4K} \cap \{|u| > 1\} | \leq \varepsilon |Q_{4K}|.$$
Let  $L=[K/A]$. 
	Then, the inequality above and our choice of $\varepsilon$ imply that 
 \begin{equation} \label{eq: 1/2}
|Q_{[L/2]}(x) \cap \{|u| \leq 1\}| \geq \frac{1}{2} |Q_{[L/2]}(x)|,
 \end{equation}
for any $x \in Q_{2K}$.

\begin{figure}[h!]
\centering
\includegraphics[scale=1.1]{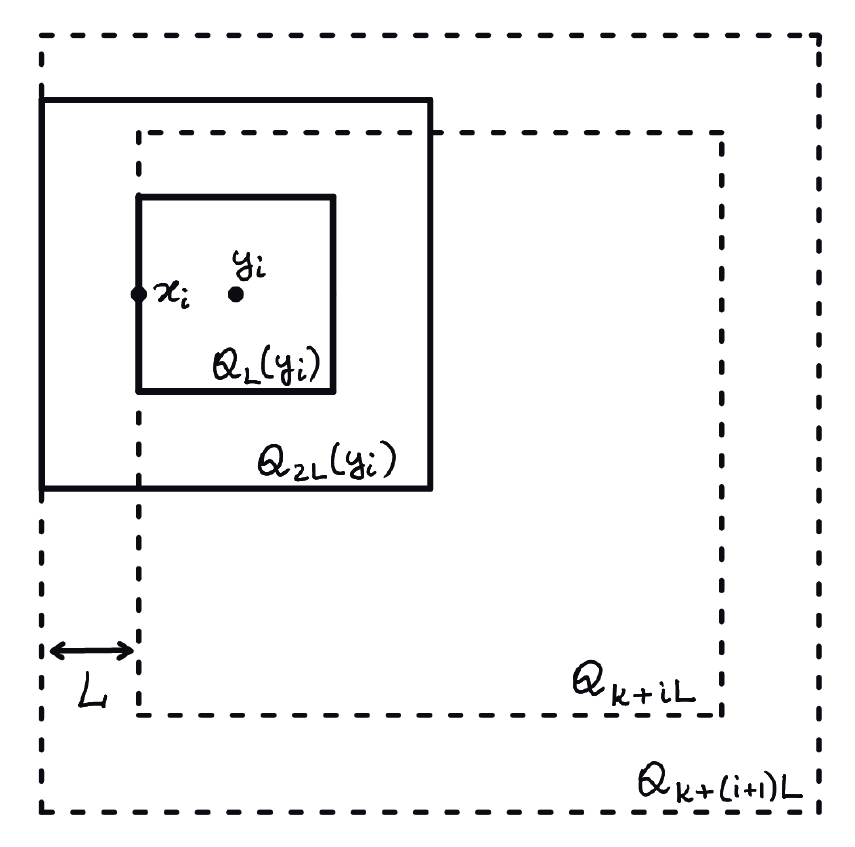}
\caption{\label{fig:prop-M}} 
\end{figure}

For $i=0,1,...,A$, denote $M(K+iL)$ by $M_i$ and fix $x_i \in Q_{K+iL}$  such that  $|u(x_i)| =M_i$. Further, we choose $y_i$ such that $x_i\in Q_{L}(y_i)\subset Q_{K+iL}$, see Figure \ref{fig:prop-M}. By \eqref{eq: 1/2} we can apply Theorem \ref{th:three} to the square $Q_L(y_i)$. This gives
$$\max\limits_{Q_{L}(y_i)}|u|  \leq C \max\limits_{Q_{2L}(y_i)}|u|^{\alpha} + C e^{-cK/A} \max\limits_{Q_{2L}(y_i)}|u|.$$
 Therefore, we have
 $$ M_i \leq C  M^\alpha_{i+1} + C e^{-cK/A} M_{i+1}.$$
 Hence, for each $i=0,1,...,A-1$,  at least one of the following inequalities holds:
 $$(i)\quad M_i \leq 2C e^{-cK/A} M_{i+1},\quad\quad 
(ii)\quad M_i \leq 2C  M^\alpha_{i+1}.$$
We note that if (i) holds for at least one $i \in [0,A-1]$, then
 $$M(2K) \geq M(K) e^{c_1K}$$
with sufficiently small $c_1$, depending on $c,A,C$ (since $N$ and therefore $K$ is large enough). This finishes the proof of the inequality \eqref{eq:propM} when (i) holds for at least one $i \in [0,A-1]$.

 Assume that (ii) holds for each $i \in [0,A-1]$.
Now, we use the assumption that $M(K)> C_1$ for some sufficiently large constant $C_1$. Denote $\alpha_1 = \frac{\alpha+1}{2}$. Note that  $\alpha_1>\alpha$ and 
$$M_i \leq2 C  M^\alpha_{i+1}\leq M^{\alpha_1}_{i+1},\quad {\text{for}}\quad i=0,...,A-1$$
provided that $C_1$ is sufficiently large. Therefore $M_0\le M_A^{\alpha_1^A}$.
Hence, by  our choice of $A$, $\alpha_1^{A}<2^{-5}$, and we  get
$$M(K)< \sqrt[32]{M(2K)}.$$
It completes the proof of the proposition.
\end{proof}
We continue to prove Theorem (B).
 We claim that the assumption $M(K)>C_1$ in Proposition \ref{pr: M} is fulfilled for $K=[N^{3/4}]$
 when $N$ is sufficiently large (and then for all larger $K$).
 Indeed, $M([\sqrt N]) \geq 2$ and there is at least one cell in $Q_{[\sqrt N]}$ where the function is 
$\le 1$. 
 Hence there are neighbors $p, p'$ in $Q_{[\sqrt N]}$ with $$|u(p) - u(p')| \geq \frac{1}{4[\sqrt N]}.$$

 The discrete gradient estimate for harmonic functions, see \cite[Theorem 14]{H} or \eqref{eq:grad} in Appendix, claims that 
 if $q$ and $q'$ are adjacent cells in $Q_R$, then 
$$|u(q)-u(q')| \leq C \max_{Q_{2R}} |u|/ R.$$ 
Hence we have
$$ M([N^{3/4}]) \ge c_2 |u(p) - u(p')| N^{3/4} \geq c_3 N^{1/4} \geq C_1$$
for $N$ large enough.

 Finally we are ready to finish the proof of Theorem (B).
 Let 
$$K_0=[N^{3/4}]+1, K_1=2K_0, \dots, K_l=2^l K_0,$$ where $K_l$ is the largest number in this sequence smaller than $N/2$. Note that $ 2^{l+3} \geq  N^{1/4}$
and hence $32^l > N$ when $N$ is large enough.
 For each $i\in[1,l]$, by Proposition 
\ref{pr: M}, we have
\begin{equation}\label{eq:M} 
M(K_i) \geq \min((M(K_{i-1}))^{32}, e^{c_1K_{i-1}}M(K_{i-1})).\end{equation}
First, we consider the case when for at least one $i \in [1,l]$ 
\begin{equation}\label{eq:M2}
M(K_i) \geq e^{c_1K_{i-1}}M(K_{i-1}).
\end{equation}
Then $M(K_i) \geq e^{c_1K_{i-1}}. $ In this case, applying Proposition \ref{pr: M} for $i+1$ in place of $i$, we get $$M(K_{i+1}) \geq \min(e^{32c_1K_{i-1}},e^{c_1K_{i}})= e^{c_1K_{i}}.$$
Repeating the argument several times, we obtain
 $$\max_{Q_{[N/2]}}|u|\geq M(K_l) \geq e^{c_1K_{l-1}}\geq e^{c_1N/8} .$$
So the proof of Theorem (B) is finished if \eqref{eq:M2} holds for at least one $i\in[1,l]$. 

Assume that it does not happen. Then by  \eqref{eq:M}
 $$M(K_i) \geq (M(K_{i-1}))^{32},\quad i=1,\dots,l.$$
This implies $$\max_{Q_{[N/2]}}|u|\geq M(K_l) \geq (M(K_0))^{32^l} \geq 2^{32^l} \geq 2^N. $$
 The proof of Theorem (B) is completed.

\section{Upper bound}\label{s:A}

\subsection{Remez inequality.}\label{ss:Remez}

 The Remez inequality compares $L^\infty$ norms of a polynomial on different subsets of  real line, we formulate the inequality in a simplified form. The sharp version is proven in the original work \cite{R}, see also \cite{B}.
 \begin{lemma}[Remez inequality]\label{l:Remez}
  Let $p$ be a polynomial of one variable of degree $d$.
  Suppose also that $I$ is a closed interval on the real line and $E$ is a measurable subset of $I$ with positive measure $|E|$. Then
 $$ \max_{ I} |p| \leq \left( \frac{4 |I|}{|E|} \right)^d \sup_{ E} |p|.$$
 \end{lemma}
 We will need a discrete version of the Remez inequality.
 \begin{corollary}\label{c:Remez} 
 Suppose that  a polynomial $p$ of degree less than or equal to $d$ is 
 bounded by $M$ at $d+l$ integer points on a closed interval $I$, then 
 $$  \max_{ I} |p| \leq \left( \frac{4 |I|}{l} \right)^d M.$$

\end{corollary}

 \begin{proof}[Proof of the corollary.]
We may assume that $p$ is not a constant function and therefore $p'$ is a non-zero polynomial of degree $\le d-1$. Let 
$$x_1< x_2< \ldots < x_{d+l}$$ 
be $d+l$ integer points on $I$ where 
$|p(x_j)| \leq M$. If $p'$ has no roots on $(x_{i},x_{i+1})$, then $|p|\leq M$ on $[x_{i},x_{i+1}]$. Since $p'$ has at most $d-1$ roots,   there are at least $l$ intervals
$(x_j, x_{j+1})$ without roots of $p'$. At the end points of such intervals   $p$ is bounded by  $M$
 and therefore $p$ is bounded by $M$ on at least $l$ disjoint  intervals of length at least one.
 So the set $\{|p|\leq M\}$ has length at least $l$ and we can apply the  Remez inequality in Lemma \ref{l:Remez}.

\end{proof}

\subsection{Some notation and a reformulation of Theorem (A)}\label{ss:sl} In this section we change our notation and introduce new coordinates which are better adjusted to our argument for the upper bound.

 We define $ s = (n+m)/2 $ and $ k = (n-m)/2 $, such that $ s,k \in \frac{1}{2} \mathbb{Z} $, and we always have $ s+k \in \mathbb{Z} $. We denote the set of all such pairs $(s,k)$ by 
$$ \mathbb{Z}^2_{\diamondtimes} = \left\{ (s,k) \in \frac{1}{2} \mathbb{Z}\times \frac{1}{2} \mathbb{Z} \, : \, s+k \in \mathbb{Z} \right\}. $$ 

\begin{figure}[h!]
\centering
\includegraphics[scale=1.1]{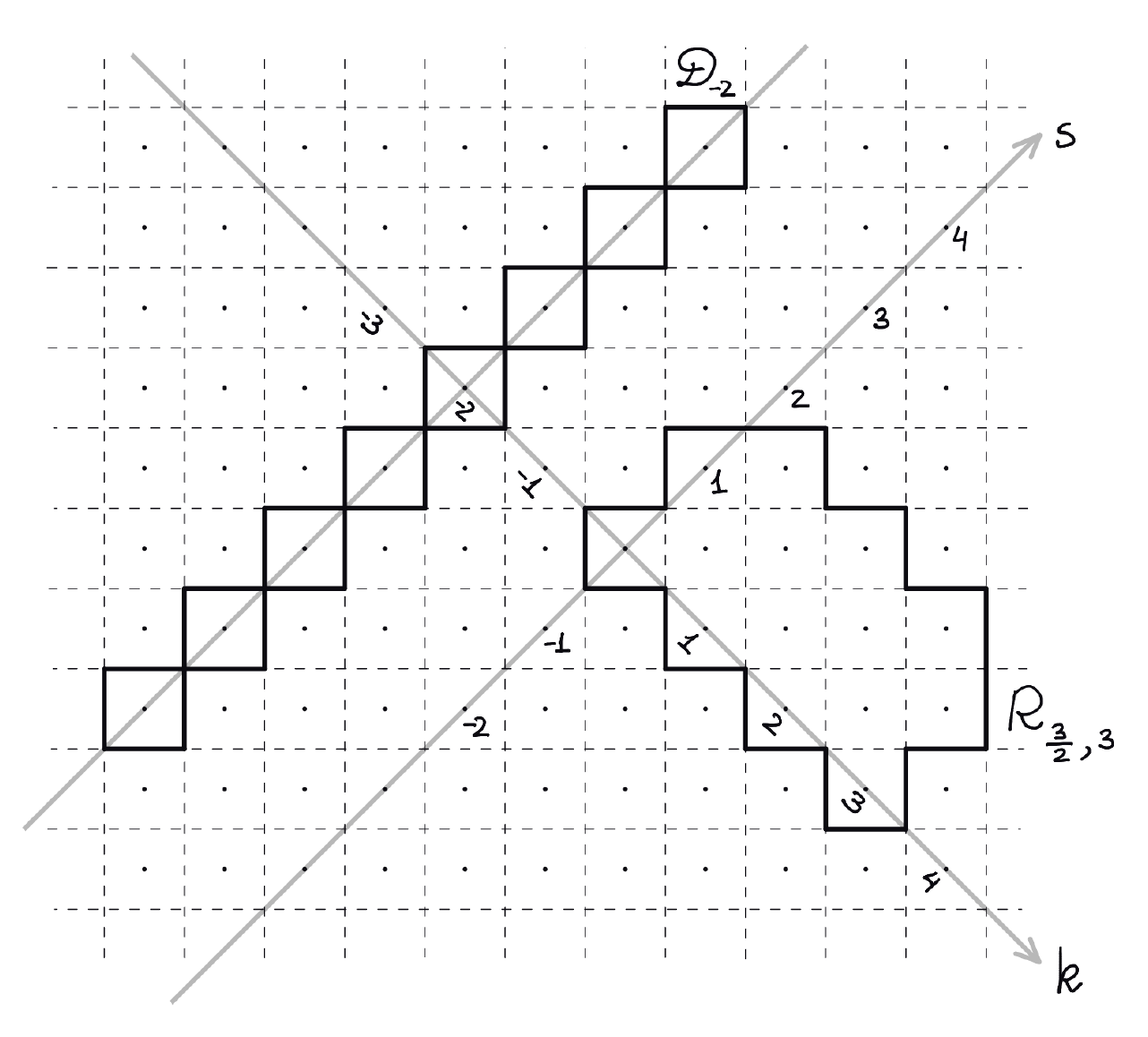}
\caption{\label{fig:intro-notations}} 
$ \mathbb{Z}^2_{\diamondtimes} $, the rectangle $ R_{3/2,2} $, and the line $ D_{-2} $
\end{figure}

For $l\in\frac 12\mathbb{Z}$, denote by $D_l$ the set
$$D_l=\{ (s,k) \in \mathbb{Z}_{\diamondtimes}^2 : k=l\}.$$
 We define a {\em rectangle} as a subset of $ \mathbb{Z}^2_{\diamondtimes} = \{ (s,k) \} $ of the form $$R=\{ a_1 \leq s \leq a_2, \,\, b_1 \leq k \leq b_2 \}$$ 
We denote it by $ R_{I,J} $, where $ I = [a_1,a_2] $  and $ J = [b_1,b_2] $  and write $R_{a,b}$ for $I=[0,a]$ and $J=[0,b]$, see Figure \ref{fig:intro-notations}.

From now on we use new coordinates $(s,k)$. Given a function $u$ on $\mathbb{Z}^2$ we identify it with the function $U$ on $ \mathbb{Z}^2_{\diamondtimes}$ defined by
\[U(s,k)=u(s+k, s-k).\] 
If $u$ is harmonic on $\Z^2$ then $U$ satisfies
\begin{equation}\label{eq:dhsl} 4U\left(s+\tfrac 12,k+\tfrac 12\right)=U(s,k)+U(s+1,k)+U(s,k+1)+U(s+1,k+1).\end{equation}
We reformulate Theorem (A) using these new notation and we do not use other coordinates till the end of the proof of Theorem (A).
We want to prove the following:
\begin{theorem*}[{\bf A$'$}] Suppose that function $U$ is defined on $\mathcal{Q}_N=R_{[-N,N],[-N,N]}$ and satisfies \eqref{eq:dhsl} for all $(s,k)\in\mathbb{Z}_{\diamondtimes}^2$ such that $-N\leq s\leq N-1, -N\leq k\leq N-1$. Assume further that $|U|\le 1$ on $(1-\varepsilon)$ portion of $\mathcal{Q}_N$ and $\varepsilon>0$ is small enough. Then 
\[|U(s,k)|\le e^{a_1(\varepsilon)N},\quad (s,k)\in\mathcal{Q}_{N/2},\]
provided that $N$ is large enough. Moreover $a_1(\varepsilon)\to 0$ as $\varepsilon\to 0$.
\end{theorem*}
Theorem (A) follows from Theorem (A$'$). Note that to deduce Theorem (A) we apply Theorem (A$'$) with a different (but comparable) value of $N$ . We cover the initial square $Q_{[N/2]}$ by several shifted new (sloped) squares and apply the statement in each of them.

\subsection{Two elementary observations}\label{ss:obs} Before we start the proof of Theorem (A$'$), we make two useful observations.

Let $I=[a_1,a_2], J=[b_1,b_2]$, where $a_1\le a_2,\  b_1\le b_2$ and  $a_1,a_2,b_1, b_2\in\tfrac{1}{2}\Z$. We consider the rectangle $R_{I,J}$ and denote by $a(R)$ and $b(R)$ its side lengths  $a(R)=a_2-a_1+1/2$ and $b(R)=b_2-b_1+1/2$.
\begin{obs} \label{obs2}
Let $U$ be any function  defined on the set
$$S=\{ (s,k) \in R_{I,J} : \min\{s-a_1,k-b_1\}\in\{0,1/2\}\} $$
 Then $U$ has a unique  discrete harmonic extension to $R=R_{I,J}$. This extension satisfies
$$\max_R |U| \leq 7^{a(R)+b(R)} \max_S |U|.$$
\end{obs}

\begin{figure}[h!]
\centering
\includegraphics[scale=1.1]{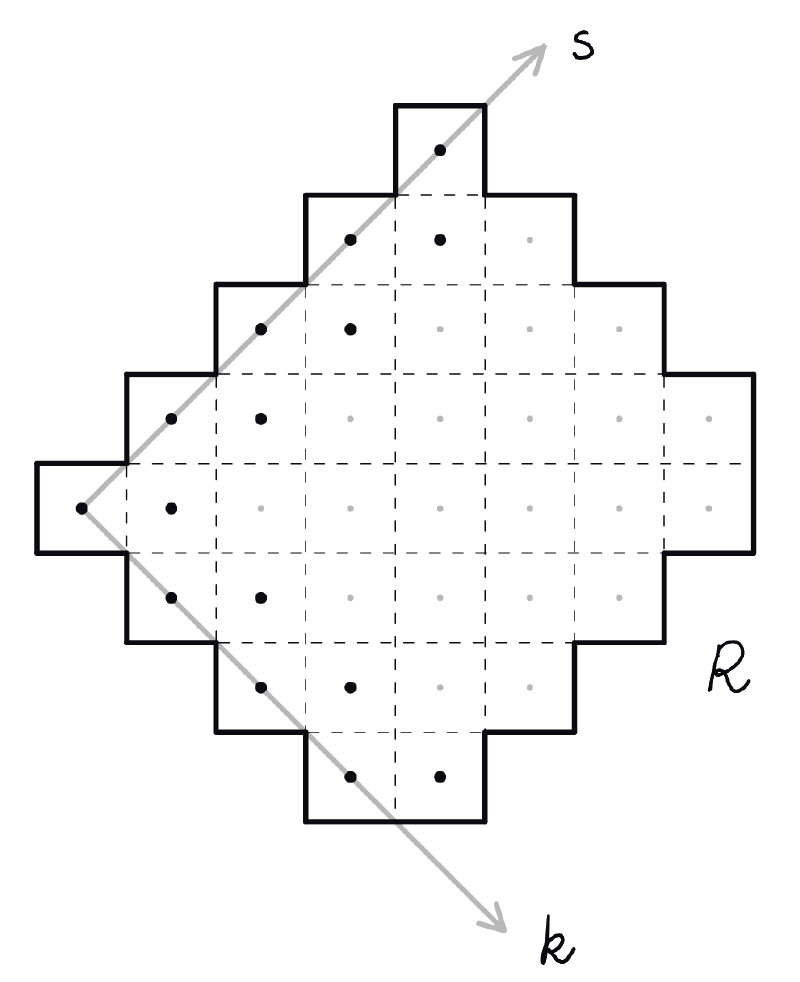}
\caption{\label{fig:obs-2}} 
The square $ R $ and its subset $ S $ (black points)
\end{figure}

\begin{proof}
 Without loss of generality we may assume $\max_S |U|=1$.
 We are going to prove that the extension is unique and satisfies 
\begin{equation}\label{eq:7}
  |U(s,k)| \leq 7^{s+k-a_1-b_1} 
\quad{\text{for\ all}}\ (s,k)\in R.
\end{equation} 
We argue by induction on $k$ and for fixed $k$ by induction on $s$. Recall that $D_l=\{(s,k) \in\mathbb{Z}^2_{\diamondtimes}  , k =l \}$ and let $T_l=D_l\cap R$. Clearly on $T_{b_1}$ and $T_{b_1+1/2}$  the function is already defined and 
the inequality \eqref{eq:7} holds.
 
Suppose that we have proved that $U$ is uniquely determined and satisfies the estimate on  $R_{k} := \mathop\cup_{b_1\le i\le k} T_i $. The value of $U(s,k+1/2)$ is prescribed at the first point of $T_{k+1/2}$,  
where $s=a_1$ or $s=a_1+1/2$ (depending on the value of $k$), and the inequality \eqref{eq:7} holds.
 Furthermore, the extension has to satisfy the mean value property \eqref{eq:dhsl} for the cells $(s,k)$ on $T_k$ with $a_1<s<a_2$,
\begin{equation*}
U(s+\tfrac 1 2,k+\tfrac 1 2) =4U(s,k)-U(s+\tfrac 1 2,k- \tfrac 1 2)- U(s-\tfrac 1 2,k-\tfrac 1 2)- U(s-\tfrac 1 2,k+\tfrac 1 2).
\end{equation*}
 Thus by induction on $s$ the values on $T_{k+1/2}$ are uniquely determined.

Furthermore, 
 since estimate \eqref{eq:7} holds on $T_k$ and $T_{k-1/2}$,
 we have 
$$| U(s+\tfrac 1 2,k+\tfrac 1 2)| \leq 6\cdot7^{s+k-a_1-b_1} + |U(s-\tfrac 1 2,k+\tfrac 1 2)|,\quad a_1<s<a_2.$$ 
So if \eqref{eq:7} holds at $(s-\tfrac 1 2,k+\tfrac 1 2)$, then it also holds at $(s+\tfrac 1 2,k+\tfrac 1 2)$. The induction argument finishes the proof. 
\end{proof}

\begin{obs}
Assume that 
a discrete harmonic function $U$ on $R=R_{I,J}$ satisfies $$ U(s,b_1)=U(s,b_1+\tfrac12) =0 , \quad a_1\leq s \leq  a_2.$$ Then 
for any $k\in[b_1+1,b_2] \cap \frac{1}{2}\mathbb{Z}$ there is a polynomial $p_k$ of degree not greater than $2(k-b_1)-2$ such that 
 \begin{equation}\label{eq:poly}
 U(s,k)= (-1)^{s+k}p_k(s), \quad a_1\leq s \leq  a_2.
\end{equation} 
\end{obs}

\begin{proof} Define the functions $p_k(s)$ by \eqref{eq:poly}. We show by induction on $k$ that $p_k$ coincides with some polynomial of degree $\le 2(k-b_1)-2$.
The basis of induction  follows from the fact that $U(s,k)=0$  for $k=b_1,b_1+1/2$ (by a polynomial of a negative degree we mean identically zero function). 
We prove the statement for $k$ assuming that it holds for $k-1/2$ and $k-1$. By the mean value property \eqref{eq:dhsl}
 for $U$, we have  
$$ p_{k}(s+1)-p_{k}(s)=-4p_{k-1/2}(s+\tfrac12)+p_{k-1}(s+1)-p_{k-1}(s),\quad a_1\leq s\leq a_2-1.$$
Then, by the induction assumption, $p_{k}(s+1)-p_{k}(s)$ coincides on $a_1\le s\le a_2-1$ with a polynomial of degree $\le 2(k-b_1)-3$. Thus 
$p_k(s)$ coincides on $a_1\le s\le a_2$ with some polynomial of degree not greater than $2(k-b_1)-2$.
\end{proof}
The next corollary is an application of the Remez inequality.
\begin{corollary} \label{cor: remez}
Assume that the rectangle $R=R_{I,J}$ satisfies $ a(R) \geq 10b(R) $. Assume also that 
a discrete harmonic function $U$ on $R$ satisfies 
$$ U(s,b_1)=U(s,b_1+\tfrac12) =0 , \quad  a_1\leq s \leq  a_2,$$
and 
$ |U(s,b_2)| \leq M$
for at least half of the points $(s,b_2)$ on $T_{b_2}=D_{b_2}\cap R$.
Then   the following inequality holds: $$ |U(s,b_2)| \leq M C^{b(R)} , \quad a_1 \leq s \leq a_2. $$
\end{corollary}
\begin{proof}
Indeed, by the observation $(-1)^{s+b_2}U(s,b_2)$ coincides with  a polynomial $p(s)$ of degree not greater than $2b(R)-3$. But the number of cells in $D_{b_2}\cap R$  is at least $a(R)-1 \geq 5(2b(R)-3)$ and $|p|\leq M$ on at least half of those cells. Applying the discrete version of the Remez inequality for $p$ we obtain a bound for $|p|$ on the interval $[a_1,a_2]$. It gives the required bound for $|u|$ on $T_{b_2}$.
\end{proof}

\subsection{Auxiliary Lemma} We will use the following lemma several times in the proof of Theorem (A).

\begin{lemma} \label{l:3d}
Let $U$ be a discrete harmonic function on a rectangle $R=R_{I,J}$ with $ a(R) \geq 10b(R).$ If 
$$ |U(s,b_1)|, |U(s,b_1+\tfrac12)| \leq M, \quad a_1 \leq s \leq a_2,$$
 and  $|U|\le M$ on at least half of the cells of   $T_{b_2}=D_{b_2}\cap R$, then 
$$ |U(s,k)| \leq M C_1^{a(R)} , \quad {\text{ for\ all}}\quad  (s,k) \in R .$$
\end{lemma}
\begin{proof}
We divide the proof into several steps.

{\it Step 1.} First, we prove the estimate $|U|\le MC_2^{a(R)}$ on $T_{b_2}$. It is enough to consider the case when $U$ is zero on the set $S_0=\{(s,k)\in R, k=b_1,b_1+1/2\}$. Indeed, we can apply Observation \ref{obs2} to find a discrete harmonic function $U_1$ in $R$, which coincides with $U$ on the set $S_0$ and for example is zero at $(s,k)\in R$ such that $s\in\{a_1,a_1+1/2\}$ and $k>b_1+1/2$. We see also that $|U_1|\le 7^{a(R)+b(R)}M\le M7^{2a(R)}$ in $R$.
Now, consider the function $U_2 = U-U_1$, which  is equal to zero on $S_0$ and is less than $(1+7^{2a(R)})M$ on at least half of the cells of $T_{b_2}$. Corollary \ref{cor: remez}, applied for $U_2$, yields the bound $|U_2|\leq C^{b(R)} (1+7^{2a(R)})M$ on $T_{b_2}$. Thus 
$|U|$ is bounded on $T_{b_2}$  by $(C^{b(R)} (1+7^{2a(R)})+7^{2a(R)})M\leq MC_2^{a(R)}$.

{\it Step 2.} Suppose that $U$ is discrete harmonic in $R=R_{I,J}$ with $a(R)\ge 10b(R)$ and that
\begin{equation} \label{eq:base}
|U(s,b_1)|, |U(s,b_1+1/2)|, |U(s,b_2)| \leq M_1\quad{\text{for\ any}}\quad  a_1 \leq s \leq a_2.
\end{equation}
Then, we prove by induction on $b(R)$, that 
\begin{equation} \label{eq:3d}
|U(s,k)| \leq M_1 9^{b(R)} \quad {\text{for}}\  
(s,k) \in \mathbb{Z}_{\diamondtimes}^2,\   a_1+b(R) \leq s \leq a_2-b(R) ,\  b_1 \leq k \leq b_2.
\end{equation}

 If $b(R) \leq 3/2$ all the values of $|U|$ are bounded by $M_1$. This is the basis of induction.
 For the induction step assume $b(R) > 3/2$. 

  Define the function $V$ on $R^-=\{(s,k)\in\mathbb{Z}^2_{\diamondtimes}: a_1\le s\le a_2-1, b_1+1/2\le k\le b_2\}$ by 
$$V(s,k)=U(s,k)+ U(s+1,k).$$
 Note that $V$ is  discrete harmonic in $R^-$ and clearly $|V|\leq 2M_1$ on $D_{b_1+1/2}\cap R^{-}$ and on $D_{b_2}\cap R^{-}$.
 We claim that $|V| \leq 6M_1$ on $D_{b_1+1}\cap R^{-}$. By the mean value property \eqref{eq:dhsl}, 
 $$V(s,b_1+1)= U(s,b_1+1)+ U(s+1,b_1+1) = 4U(s+1/2,b_1+1/2) - U(s,b_1)- U(s+1,b_1).$$
 Since $|U|\leq  M_1$ on $D_{b_1} \cap R $ and on $D_{b_1+1/2} \cap R $, the identity above implies
 $|V| \leq 6M_1$ on  $D_{b_1+1} \cap R^{-}$.

  Now, we know that $|V| \leq 6M_1$ on three lines:  $D_{b_1+1/2}\cap R^{-} , D_{b_1+1} \cap R^{-}$ and 
 $D_{b_2} \cap R^{-}$. We are in position to apply the induction assumption for  $V$ and  $R^-$. It gives
 $$ |V(s,k)| \leq 6M_1 9^{b(R^-)} ,\quad  a_1+b(R^-) \leq s \leq a_2-1-b(R^-),\ b_1+1/2\leq k \leq b_2.$$
Applying \eqref{eq:dhsl} once again, we get 
$$4U(s,k)= V(s-1/2,k-1/2) + V(s-1/2,k+1/2).$$
 For every $(s,k) \in \mathbb{Z}_{\diamondtimes}^2$,  $a_1+ b(R) \leq s \leq a_2-b(R) $, $ b_1+1 \leq k \leq b_2-1/2,$ it yields $$|U(s,k)| \leq M_1 9^{b(R)}.$$
 While on $D_{b_1} \cap R$, $D_{b_1+1/2}\cap R$,  and $D_{b_2} \cap R$ the function $|U|$ is smaller than $M_1$ by the initial assumption.  The induction step is completed.
We therefore have proved  \eqref{eq:3d}. 

{\it Step 3} Finally, applying Observation \ref{obs2} to rectangles $R_{[a_1,a_1+b(R)+1/2], J}$ and $R_{[a_2-b(R)-1/2,a_2],J}$,
  we obtain 
\[|U(s,k)|\le M_1 9^{b(R)}7^{b(R)+1}\le M_1C_3^{b(R)}, {\text{ for\ all}}\quad  (s,k) \in R_{I,J}.\]

Now we combine the steps. The first step implies \eqref{eq:base} with $M_1=MC_2^{a(R)}$, then Steps 2 and 3 give  $|U(s,k)|\le M_1C_3^{b(R)}\le MC_2^{a(R)}C_3^{b(R)}\le MC_1^{a(R)}$ for all $(s,k) \in R$. 
This completes the proof of Lemma \ref{l:3d}.
\end{proof}

\subsection{Good rectangles}
We make the last preparation for the proof of Theorem (A). Let $\mathcal{Q}=\mathcal{Q}_N=R_{[-N,N],[-N,N]}$ as in Section \ref{ss:sl}. We fix a function  $ U : \mathcal{Q} \rightarrow \mathbb{R} $ which is  discrete harmonic and such that $ |U| \leq 1 $ on $(1-\varepsilon)$ portion of $\mathcal{Q}$.
 
We consider  "good  rectangles'' $R_{[a_1,a_2],[b_1,b_2]}=\{(s,k)\in\mathbb{Z}^2_{\diamondtimes}: a_1\le s\le a_2, b_1\le k\le b_2\}$, whose side-lengths $a(R)=a_2-a_1+1/2$ and $b(R)=b_2-b_1+1/2$ are comparable, and on which the function $U$ is not  large.  More precise definition is below. 
\begin{definition} Let  $A=C_1^{800}$, where $C_1$ is the constant from Lemma \ref{l:3d}.
 A rectangle $ R\subset \mathcal{Q} $ is called {\em good}, if   \\
\[ a(R)/10 \leq  b(R) \leq 10 a(R)\quad{\text{and}}\quad  
\max_R|U| \leq A^{a(R)+b(R)}.\]
\end{definition}

The following lemma helps to expand good rectangles. It claims that if there is a good rectangle and near this rectangle the portion of cells with $|U|>1$ is small, then one can find a new larger good rectangle that contains the old one. For simplicity we formulate it for rectangles $R_{I,J}$ with $I=[0,a]$ and $J=[0,b]$ but will apply it for general rectangles, the proof is the same up to small changes of notation.

\begin{lemma}  \label{l:ext}
 Assume that  $R_{a,b}$  is a good  rectangle with $a\geq b \geq 1/2$ and that the number of cells in $R_{a,3b}$ where  $|U|>1$ 
 is less than $\frac{1}{10^5}|R_{a,b}|$. Then  for  any $b' \in [3b/2, 2b]$, the rectangle $R_{a,b'}$ is also good.
   \end{lemma}
\begin{proof} If $b \leq 40$, then $a \leq 400$ because $R_{a,b}$ is good. In this case $|R_{a,4b}| \leq 10^5$ and
$|U| \leq 1$ everywhere on $R_{a, 4b}$.

 Assume that $b > 40$. 
 For each $k=1,2, \dots, 40$ we can choose  a number $$b_k \in \frac{1}{2}\mathbb{Z} \cap (b(1+(2k-1)/40),b(1+2k/40))$$ such that at least half of
cells $z$ on the line $D_{b_k}\cap R_{a,3b}$ satisfy $|U(z)| \leq 1$.
 Define $b_0=b$.
 Note that $$ b/40 \leq b_{k}-b_{k-1} \leq 3b/40 \leq 3a/40 \leq a/10 - 1/2 $$ and $b_{40} \in (2b, 3b)$.

 Denote by $M_k=\max_{R_{a,b_k}}|U|$.
 It suffices to show that $M_{40} \leq A^{a + \frac{3}{2}b +1}.$  Indeed, it implies that
$$\max_{R_{a,b'}}|U|\le \max_{R_{a,2b}}|U|\le M_{40} \le A^{a+\frac{3}{2}b+1}\le A^{a+b'+1},$$
and thus that $R_{a,b'}$ is good.

 Since $R_{a,b}$ is good, we know that $M_0 \leq A^{a+b+1}$. 
We show that 
$$M_{k+1}\leq M_k C_1^{a}\quad{\text{for}}\ k=0,1,2, \dots 40.$$ 
Consider the rectangle 
$$R_k=\{(s,t) \in \mathbb{Z}^2_{\diamondtimes} : 0 \leq s \leq a, b_k -1/2 \leq t \leq b_{k+1}\}$$
and note that $b_{k+1} -b_{k}+1/2 \leq  a / 10$. We have $|U| \leq M_k$ on $R_k\cap D_{b_k-1/2}$ and $R_k\cap D_{b_k}$ and also $|U|\leq 1$ on half of cells of  $R_k\cap D_{b_{k+1}}$. We therefore can apply Lemma \ref{l:3d}. We get  
 $ |U|\leq M_k C_1^{a}$
  on $R_{k+1}$ and hence
$$ M_{k+1}\leq M_k C_1^{a}.$$

 Recalling that $M_0 \leq A^{a + b+1}$ and using the inequality above consecutively for $k=0,1,...,39$,
 we get 
$$M_{40} \leq A^{a+b+1} C_1^{40 a} \leq A^{a+\frac{3}{2}b+1},$$
 provided that $A \geq C_1^{800}$.
\end{proof}

Given a square $R=R_{I,J}$ with $I=[a_1,a_1+l]$, $J=[b_1,b_1+l]$ and $a(R)=b(R)=l+1/2$ and an odd integer $n$, we denote by $nR$ the square with the same center and side length $a(nR)=b(nR)=n(l+1/2)$.
\begin{corollary} \label{pr: max-rect}
If $ R \subset \mathcal{Q} $ is a  good  square and  $ 9R \subset \mathcal{Q} $, then either 
$$|9R\cap\{|U| > 1 \}|\ge \frac{1}{10^5}|R|,$$ or $3R$ is  good.
\end{corollary}

\begin{figure}[h!]
\centering
\includegraphics[scale=0.7]{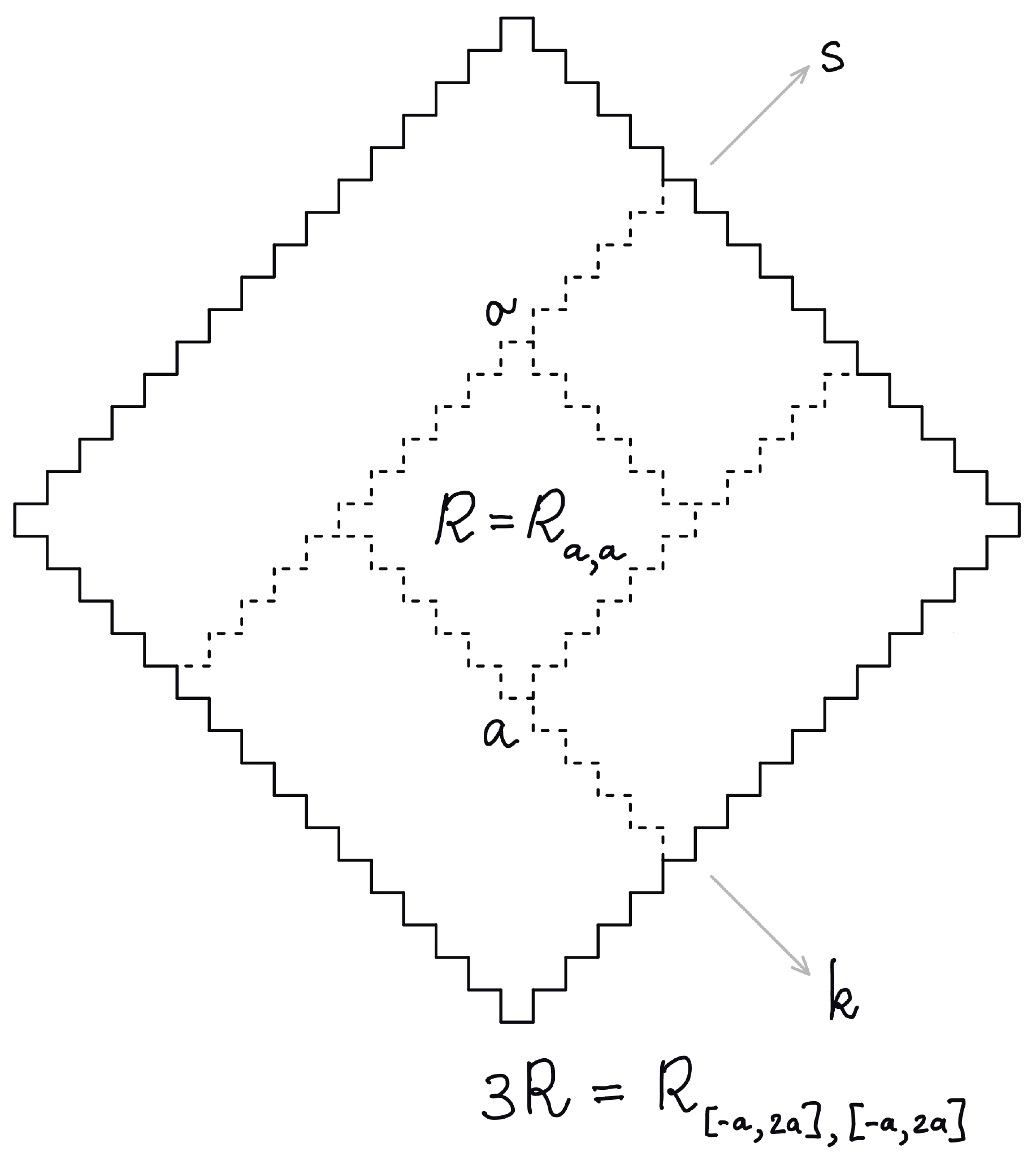}
\caption{\label{fig:cor-max-rect}} 
\end{figure}

\begin{proof} 
 Assume that the number  of cells $ z $ in $ 9R $ with $ |U(z)| > 1 $ is less than $ \frac{1}{10^5}|R|$.
 Then applying (rotated and shifted versions of) Lemma \ref{l:ext}  four times, we can expand $R$ and show that $3R$ is good. If we shift the initial square to be $R_{a,a}$ then we obtain the following sequence of good rectangles $R_{a,a}, R_{2a,a}, R_{2a,2a}, R_{[-a,2a],[0,2a]}$ and finally $3R=R_{[-a,2a],[-a,2a]}$.
\end{proof}

\subsection{Maximal good squares and a covering lemma}
For the proof of Theorem (A$'$) we will consider good squares that are maximal with respect to inclusion and therefore the
portion of cells  near these squares with $|U|>1$ is not too small.
\begin{definition}
 We call a good  square $ R \subset \mathcal{Q}_N $ {\em maximal} for  $\mathcal{Q}_N$ if there is no good square $ R'$ such that   $ R \subsetneq R'\subset \mathcal{Q}_N $. 
\end{definition}

 Now, we formulate a proposition to be used in the proof of Theorem (A$'$).

\begin{prop}   \label{prop:main}
 Suppose that
$$| \mathcal{Q}_K \cap \{|U| > 1\} | \leq 10^{-20} |\mathcal{Q}_K|.$$ Then
there is a good square $R$ such that 
$$ \mathcal{Q}_{[K/100]} \subset R \subset \mathcal{Q}_{K}.$$
\end{prop}
\begin{proof} 
 Consider the collection $\mathbb{M}$ of all maximal for $\mathcal{Q}_{[K/10]}$ squares that contain at least one cell in $\mathcal{Q}_{[K/100]}$.  
 Note that the total number of cells in all maximal squares satisfies 
\begin{equation} \label{eq:10}
 |\bigcup\limits_{\mathbb{M}}  R| \geq \frac{1}{2} |\mathcal{Q}_{[K/100]}| \geq \frac{1}{10^5} |\mathcal{Q}_{K}|.
\end{equation}
 This is true because each cell $z\in \mathcal{Q}_{[K/100]}$ with $|U(z)| \leq 1$ is contained in some maximal square and   $\{ |U|\leq 1\}$ occupies at least a half of $\mathcal{Q}_{[K/100]}$.

We consider two cases:\\
 (1)  $a(R_0)\geq K/50$ for some $R_0\in \mathbb{M}$, 
(2)
$a(R)\leq K/50$ for every $R\in\mathbb{M}$.

We show that in the first case the conclusion of Proposition \ref{prop:main} holds and that the second case never occurs.

First, suppose that there is $R_0\in\mathbb{M}$ with $a(R_0)\geq K/50.$ Since $R_0\subset \mathcal{Q}_{[K/10]}$, we have $9R_0\subset \mathcal{Q}_K$.  Further,
$$|\{ |U|>1 \}\cap 9 R_0 |\leq|\{ |U|>1 \}\cap \mathcal{Q}_{K} | \leq 10^{-20} |\mathcal{Q}_{K}| \leq 10^{-7} |R_0|.$$ 
Applying   Corollary \ref{pr: max-rect} for $R_0$ we conclude that $3R_0$ is also good. 
Since $R_0$ intersects $\mathcal{Q}_{[K/100]}$ and $a(R_0)>K/50$, we see that the good square $R=3 R_0 $ contains  $\mathcal{Q}_{[K/100]}$ and  the proposition is proved in the first case.

 For the second case, we have $a(R)\le K/50$ for each $R\in\mathbb{M}$. Consider any $R\in\mathbb{M}$. Since $R$ intersects $\mathcal{Q}_{[K/100]}$ we see that $3R\subset \mathcal{Q}_{[K/10]}$.
Thus, by the maximality of $R$ in $\mathcal{Q}_{[K/10]}$, $3R$ is not good. Then Corollary \ref{pr: max-rect} implies
\begin{equation} \label{eq:4r}
 |9R \cap \{ |U|>1\}| \geq \frac{1}{10^5} |R|.
\end{equation}

We will use  the following Vitali-type covering lemma:\\
{\it Given a finite collection $\mathcal{M}=\{q_j\}$ of squares with sides parallel to the coordinate axis, there exists a subcollection $\mathcal{M}'=\{q_{j_k}\}$ such that $q_{j_k}$ are pairwise disjoint and $$\bigcup_{\mathcal{M}'}3q_{j_k}\supset \bigcup_{\mathcal{M}}q_j.$$}\\
The statement is simple and is proved by selecting the largest possible square on each step such that the chosen subcollection remain disjoint, we refer the reader to \cite[Chapter 1]{G}. A similar standard argument for balls in $\R^n$ can be found for example in \cite{T}.

Let $\mathbb{M}$ be the collection of  maximal squares as above. We apply the covering lemma to the collection of squares $9 R$ for $R\in\mathbb{M}$. (Note also that $9R\subset\mathcal{Q}_K$ for each $R\in\mathbb{M}$.) 
There exists a subcollection $\mathbb{M}' \subset \mathbb{M}$ such that 
\[ \mathop\bigcup\limits_\mathbb{M'}  27R  \supset \mathop\bigcup\limits_\mathbb{M}  9R\]
and 
$9R_1$ and $9R_2$ are disjoint for any distinct $R_1,R_2 \in \mathbb{M'}$.  Then we get 
\[ \sum_{\mathbb{M}'}|R|  \geq \frac{1}{10^3} |\mathop\bigcup\limits_\mathbb{M}  R|\]
and, since $\{9R, R\in\mathbb{M}'\}$ are disjoint,  \eqref{eq:4r}  implies
$$|\{ |U|>1\} \cap \mathcal{Q}_K | \geq \sum_{R\in\mathbb{M}'} |9R \cap \{ |U|>1\}| \geq \sum_{R\in\mathbb{M}'} \frac{1}{10^5} |R|.$$
Thus 
\[|\{ |U|>1\} \cap \mathcal{Q}_K | > \frac{1}{10^{10}} |\mathop\bigcup\limits_\mathbb{M}  R|\]
and by \eqref{eq:10}
$$|\{ |U|>1\} \cap \mathcal{Q}_K | \geq \frac{1}{10^{15}}|\mathcal{Q}_K|.$$ 
This contradicts the assumption of the proposition, hence the second case never occurs. Therefore we can always cover $\mathcal{Q}_{[K/100]}$ by a good rectangle.
\end{proof}

\subsection{Proof  of Theorem (A$'$)}
By the assumption 
$$|\mathcal{Q}_N\cap\{|U|\leq 1 \} \}| \geq (1-\varepsilon)|\mathcal{Q}_N|.$$
Our goal is to show that if $\varepsilon>0$ is sufficiently small and $N$ is sufficiently large, $N\geq N(\varepsilon)$, then
 $$ \max\limits_{\mathcal{Q}_{N/2}} |U| \leq  e^{a_1N}, $$
where  
 $ a_1=a_1(\varepsilon) \to 0$ as $\varepsilon \to 0$.

 Let $K< N/200$ and cover $\mathcal{Q}_{N/2}$ by $\simeq ([N/K])^2$ squares  $\mathcal{Q}_{K}(x_i)$ such that 
 $\mathcal{Q}_{100K}(x_i) \subset \mathcal{Q}_N$.  The number of cells in $\mathcal{Q}_{100K}(x_i)$ with $|U|>1$ is less than $\varepsilon|\mathcal{Q}_N|$.
 Hence
$$ |\mathcal{Q}_{100K}(x_i)\cap\{|U|> 1 \} \}| \leq  \varepsilon \left(\frac{N}{K}\right)^2 |\mathcal{Q}_K|.$$
 If we assume that $K=cN$ and   $\varepsilon  \leq c^2/10^{20}$, then by Proposition \ref{prop:main}, for each $i$, 
 there is a good square $R_i$ such that $\mathcal{Q}_K(x_i) \subset R_i \subset \mathcal{Q}_{100K}(x_i)$. By the definition of a good square we have
$$\max_{R_i}|U| \leq A^{200K} =e^{N 200 c\log A  }$$ 
Thus $|U|\le e^{N 200 c\log A  }$ in each $\mathcal{Q}_K(x_i)$  and therefore in $\mathcal{Q}_{N/2}$. We had proved the first part of Theorem (A$'$).

 To prove that $a_1(\varepsilon)\to 0$ as $\varepsilon\to 0$, we fix any $a_1>0$  and choose $K=K(N,a_1,\varepsilon)$ so that  $$A^{200K} \leq e^{a_1N} \  \textup{ and }\  \varepsilon \left(\frac{N}{K}\right)^2 \leq 1/10^{20},$$  and we can make such a choice  if $\varepsilon$ is sufficiently small.

\section*{Appendix}
\setcounter{section}{1}
\setcounter{subsection}{0}
\setcounter{theorem}{0}
\renewcommand{\thesection}{\Alph{section}}
The aim of the appendix is to prove Theorem \ref{th:three}. The proof is a  modification of the one given in \cite{GM}. First, we write down explicit formulas for the discrete Poisson kernel and prove an estimate for its analytic continuation into the complex plane, as it was done in \cite{GM}. Then we apply polynomial approximation and the discrete version of the Remez inequality to finish the proof. Note that we return to  the standard lattice $\mathbb{Z}^2$ and the notation used in the first part of the text.

\subsection{The discrete Poisson kernel}
 For each integer $k\in (0,2N)$ we define $a_k$ to be the only positive solution of the equation
\[
\cosh  a_k=2-\cos \frac{k\pi }{2N}.\]
Then 
\[
f_k(n,m)=\sin \left(\frac{\pi k n}{2N}\right)\sinh( a_k m)\]
is a discrete harmonic function. Now we fix an integer $n_1\in[-N+1,N-1]$ and consider a discrete harmonic function of $(n,m)$
\begin{equation}\label{eq:Pker}
F(n,m)=\frac{1}{N}\sum_{k=1}^{2N-1} \sin \left(\pi k \frac{n+N}{2N}\right)\sin \left(\pi k\frac{n_1+N}{2N}\right)\frac{\sinh a_k(m+N)}{\sinh 2a_kN}.
\end{equation}
Clearly $F(-N,m)=F(N,m)=F(n,-N)=0$. Furthermore, by the orthogonality identities for discretized trigonometric functions, we have $F(n_1,N)=1$ and $F(n,N)=0$ when $n\neq n_1, -N< n< N$. Thus $F(n,m)$ is the discrete Poisson kernel for the domain $Q_N$ at the boundary point $y=(n_1,N)$. We denote it by $P(x,y)$, where $x=(n,m)$. Poisson kernel on the three other sides of the square can be computed in a similar way. We define the boundary of $Q_N$ by
\[\partial Q_N=\{(n_1, m_1)\in Q_N: \max\{|n_1|,|m_1|\}=N,\ |n_1|\neq|m_1|\},\]
it consists of the four sides of the square without the corners, denote the set of these four corners by $K_N$. The values of a discrete harmonic function on $Q_N\setminus K_N$ are defined by its values on $\partial Q_N$.  More precisely, for any discrete harmonic function $u$ in $Q_N$, we have
\begin{equation}\label{eq:Poisson}
u(x)=\sum_{y\in\partial Q_N} P(x,y)u(y),\quad x\in Q_N\setminus K_N.\end{equation}

We need the following  statement.

\begin{lemma}\label{pr:ext}
For any $y\in\partial Q_N$ and any integer $m\in[-N/2,N/2]$ the function $g(t)=P((tN,m),y)$  has a holomorphic extension on \[\Omega=\{z: |{\rm{Re}} (z)|\le 1/2, |{\rm{Im}} (z)|\le 1/16\}\subset \C,\] this extension  satisfies $|g(z)|\le CN^{-1}$ for all $z\in \Omega$.
\end{lemma}

\begin{proof} The holomorphic extension is given by \eqref{eq:Pker}. We want to prove the estimate.
Let $y=(n_1,m_1)$, we consider two cases: $|n_1|=N$ and $|m_1|=N$. 
We note that 
 $\cosh a_k=2-\cos\frac{k\pi}{2N}$, and we claim that $a_k\ge \frac{k}{2N}$. First when $N\le k\le 2N$, we have
\[\cosh a_k=2-\cos\frac{k\pi}{2N}\ge 2\]
and thus $a_k\ge 1\ge \frac{k}{2N}$ because $\cosh 1=(e+e^{-1})/2<2$.

For $0\le k\le N$ we use the inequality $\cos x\le 1-x^2/\pi$ when $0\le x\le \pi/2$ and  we obtain
\[\cosh a_k\ge 1+\frac{k^2\pi}{4N^2}\] which implies
\[2\sinh^2 a_k/2=\cosh a_k-1\ge \frac{k^2\pi}{4N^2}\ge \frac{k^2}{2N^2}.\]
We have that $\sinh 0=0$, $(\sinh t)'=\cosh t$ and $\cosh t$ is increasing when $t>0$. Then $\sinh t/2\le t/2\cosh 1\le t$ for $t\le 1$. Taking $t=k/(2N)$ we see that $\sinh a_k/2\ge t\ge \sinh t/2$ and
\[ a_k\ge \frac{k}{2 N}.\]

Then, for the first case we have,
\begin{align*}
|g(z)|&=|P((zN,m),(\pm N,m_1))|\\
&\le \frac{1}{N}\sum_{k=1}^{2N-1}\left|\frac{\sinh a_kN(z+1)}{\sinh2a_kN}\right|\\
&\le \frac{1}{N}\sum_{k=1}^{2N-1}\frac{\cosh 3a_kN/2}{\sinh2a_kN}\quad\quad(|{\rm{Re}}(z)|\le 1/2)\\
&\le \frac{1}{N}\sum_{k=1}^{2N-1}\frac{1}{\sinh a_kN/2}\quad\quad(\sinh 2t\ge \cosh 3t/2\sinh t/2)\\&\le \frac{C}{N} \quad\quad(a_kN/2\ge k/4). 
\end{align*} 

Similarly, for the case $|m_1|=N$, we get
\begin{align*}
|g(z)|&=|P((zN,m),(n_1,\pm N)|\\
&\le \frac{1}{N}\sum_{k=1}^{2N-1}\left|\sin\frac{\pi k(z+1)}{2} \right|\frac{\sinh (3a_kN/2)}{\sinh2a_kN}\\
&\le
\frac{1}{N}\sum_{k=1}^{2N-1}\cosh\frac{\pi k}{32} \frac{\sinh (3a_kN/2)}{\sinh2a_kN}\quad\quad(|{\rm{Im}}(z)|\le 1/16)\\
&\le
\frac{1}{N}\sum_{k=1}^{2N-1}\cosh\frac{a_kN}{4} \frac{\sinh (3a_kN/2)}{\sinh2a_kN}\quad\quad(\pi k/32\le k/8\le a_kN/4)\\
&\le \frac{1}{N}\sum_{k=1}^{2N-1}\frac{\sinh (7a_kN/4)}{\sinh2a_kN}\quad\quad(\sinh(s+t)\ge \cosh s\sinh t,\ s,t>0)\\
&\le \frac{1}{N}\sum_{k=1}^{2N-1}\frac{1}{\cosh a_kN/4}\\
&\le \frac{C}{N}\quad\quad(a_kN/4\ge k/8).
\end{align*}
\end{proof}

\begin{corollary}\label{l:ap1} Suppose that $u$ is a discrete harmonic function on $Q_N$ such that $|u|\le M$ on $Q_N$ and that $m_0$ satisfies $|m_0|<N/2$. Then there exists an analytic function $f=f_{m_0}$ defined in a neighborhood of $\Omega=\{z=s+it, |s|\le 1/2, |t|\le 1/16\}$ such that  $u(n,m_0)=f(n/N)$ and $|f|\le C_0M$ in $\Omega$.
\end{corollary}

The corollary follows immediately from Lemma \ref{pr:ext} and the Poisson representation formula \eqref{eq:Poisson}. We will use this quantitative analyticity of $f$ for the following estimate. Let
$P_d$ be the Taylor polynomial of $f$ of degree $d$ centered at the origin. Then the standard Cauchy estimate implies 
\begin{equation}\label{eq:anal}
|f(z)-P_d(z)|\le \max_{\Omega}|f| (L|z|)^{d+1},\quad |z|<1/L,
\end{equation} for some $L=L(\Omega)$.

We remark also that the corollary above, combined with the Cauchy estimates for derivatives of analytic functions, implies  that if $q,q'\in Q_{N/3}$ are two neighboring cells then
\begin{equation}\label{eq:grad}
|u(q)-u(q')|\le CN^{-1}\max_{Q_N}|u|
\end{equation}
for any discrete harmonic function $u$ in $Q_N$. This in turn implies  the gradient estimate that we used in the proof of Theorem (B).

\subsection*{Proof of Theorem \ref{th:three}}

To prove Theorem \ref{th:three}, it is enough to prove the following statement for some $\gamma\in(0,1/2)$:\\
 {\it if $u$ is a discrete harmonic on $Q_N$  are such that  $|u|<M$ on $Q_N$ and $|u|<\sigma$ in at least a half of the cells in  $Q_{[\gamma N]}$, then $|u|<C(M^\beta\sigma^{1-\beta}+e^{-cN}M)$ on $Q_{2[\gamma N]}$}. 

Indeed, the statement of the theorem follows from this one by iterations. First we divide the square $Q_{N/4}$ into ones with the side length $\approx \gamma N/2$, find one of such squares where  $|u|<1$ in at least half of the cells. Then we iterate the estimate finitely many times using the following observation: if  we fix $M$ and for $\sigma,c,C>0$ and $\beta\in(0,1)$ define
$r(\sigma,\beta,C,c)=C(M^\beta\sigma^{1-\beta}+e^{-cN}M)$
then 
\begin{equation}\label{eq:longr}
r(r(\sigma,\beta_1,C_1, c_1),\beta,C,c)\le r(\sigma, \beta+\beta_1-\beta\beta_1, C(C_1^{1-\beta}+1), \min\{c_1(1-\beta),c\}).\end{equation}
The last inequality can be shown by an elementary computation which we skip.

We will prove the statement when $\gamma<2^{-8}L^{-1}$, where $L$ is the constant from \eqref{eq:anal}. First we consider an integer $m\in [-\gamma N, \gamma N]$ such that $|u(n,m)|<\sigma$ for at least quarter of integers $n\in[-\gamma N,\gamma N]$ and propagate the estimate in the  horizontal direction.

By Corollary \ref{l:ap1} there is an analytic function $f$ in $\Omega$ bounded by $C_0M$ and such that  $f(n/N)=u(n,m)$ when $|n|<N/2$. By the assumption of the statement, there exist points $s_j\in[-\gamma,\gamma]$, $j=1,..., J$ such that $|f(s_j)|\le \sigma$ and $J\ge \frac12\gamma N$,  and $|s_j-s_i|\ge 1/N$ when $i\neq j$.

We consider now two cases, (i) $C_0M(2\gamma L)^{[J/2]}<\sigma$ and (ii) $C_0M(2\gamma L)^{[J/2]}\ge \sigma$.

(i) We choose $J_0\le [J/2]$ such that $2\gamma L\sigma<C_0M(2\gamma L)^{J_0}<\sigma$ and let $P$ be the Taylor polynomial of $f$ of degree $J_0-1$. The inequality $|f(s_j)|<\sigma$ and the estimate \eqref{eq:anal} imply $|P(s_j)|\le 2\sigma$ for $j=1,...,J$. Then, by a normalized version of Corollary \ref{c:Remez}, we obtain
\[|P(s)|\le \left(\frac{16N\gamma}{J-J_0+1}\right)^{J_0}2\sigma,\quad s\in[-2\gamma, 2\gamma].\]
We have $J-J_0+1\ge J/2\ge \frac{\gamma}{4}N$ which implies $|P(s)|\le 64^{J_0-1}2\sigma$.
Furthermore, using the inequalities \eqref{eq:anal} and  $64<(4\gamma L)^{-1}$ , we get
\[|f(s)|\le |P(s)|+C_0M(2L\gamma)^{J_0}\le 64^{J_0-1}2\sigma+\sigma<2^{2-J_0}C_0M+\sigma,\quad s\in[-2\gamma, 2\gamma].\] Thus $\max_{[-2\gamma,2\gamma]}|f|\le CM^{\beta}\sigma^{1-\beta}$.

(ii) If $\delta=C_0M(2\gamma L)^{[J/2]}\ge \sigma$ then we approximate $f$ by the  Taylor polynomial $P$ of degree $[J/2]-1$. By \eqref{eq:anal} and the inequality $|f(s_j)|<\sigma$, we have $|P(s_j)|\le 2\delta$. Then, applying Corollary \ref{c:Remez} once again, we get $|P(s)|\le 64^{[J/2]-1}2\delta$ on $[-2\gamma,2\gamma]$ and 
\[|f(s)|\le 64^{[J/2]-1}2\delta+C_0M(2L\gamma)^{[J/2]}\le 2C_0M(2^7\gamma L)^{[J/2]}\le Ce^{-cN}M. \]
Thus we conclude that $|u(n,m)|<r(\sigma,\beta,C,c)$ for $m$ chosen as above and all integer $n\in[-2\gamma N, 2\gamma N]$. Note that the number of horizontal lines on which we did propagation is at least one quarter of the integers in $[-\gamma N, \gamma N]$. Now we repeat the argument and propagate smallness from horizontal lines to each vertical one and apply \eqref{eq:longr} again. This completes the proof of the statement and of Theorem \ref{th:three}.

\end{document}